\documentclass[12pt]{article}
\usepackage{amsfonts,amsthm, amsmath}
\usepackage{epsfig}
\usepackage{makeidx}
\usepackage{graphicx}

\makeindex
\def\E{{\mathbb E}}
\def \P{{\mathbb P}}

\newenvironment{namelist}[1]{%
\begin{list}{}
{

\settowidth{\labelwidth}{#1}
\setlength{\leftmargin}{1.1\labelwidth}
}
}{%
\end{list}}
\newcommand{\ncom}{\newcommand}
\ncom{\ul}{\underline}
\ncom{\beq}{\begin{equation}}
\ncom{\eeq}{\end{equation}}
\ncom{\bea}{\begin{eqnarray*}}
\ncom{\eea}{\end{eqnarray*}}
\ncom{\beqa}{\begin{eqnarray}}
\ncom{\eeqa}{\end{eqnarray}}
\ncom{\nno}{\nonumber}
\ncom{\non}{\nonumber}
\ncom{\ds}{\displaystyle}
\ncom{\half}{\frac{1}{2}}
\ncom{\mbx}{\makebox{.25cm}}
\ncom{\hs}{\mbox{\hspace{.25cm}}}
\ncom{\rar}{\rightarrow}
\ncom{\Rar}{\Rightarrow}
\ncom{\noin}{\noindent}
\ncom{\bc}{\begin{center}}
\ncom{\ec}{\end{center}}
\ncom{\sz}{\scriptsize}
\ncom{\rf}{\ref}
\ncom{\s}{\sqrt{2}}
\ncom{\sgm}{\sigma}
\ncom{\Sgm}{\Sigma}
\ncom{\psgm}{\sigma^{\prime}}
\ncom{\dt}{\delta}
\ncom{\Dt}{\Delta}
\ncom{\lmd}{\lambda}
\ncom{\Lmd}{\Lambda}
\ncom{\Th}{\Theta}
\ncom{\e}{\eta}
\ncom{\eps}{\epsilon}
\ncom{\pcc}{\stackrel{P}{>}}
\ncom{\lp}{\stackrel{L_{p}}{>}}
\ncom{\dist}{{\rm\,dist}}
\ncom{\sspan}{{\rm\,span}}
\ncom{\re}{{\rm Re\,}}
\ncom{\im}{{\rm Im\,}}
\ncom{\sgn}{{\rm sgn\,}}
\ncom{\ba}{\begin{array}}
\ncom{\ea}{\end{array}}
\ncom{\hone}{\mbox{\hspace{1em}}}
\ncom{\htwo}{\mbox{\hspace{2em}}}
\ncom{\hthree}{\mbox{\hspace{3em}}}
\ncom{\hfour}{\mbox{\hspace{4em}}}
\ncom{\vone}{\vskip 2ex}
\ncom{\vtwo}{\vskip 4ex}
\ncom{\vonee}{\vskip 1.5ex}
\ncom{\vthree}{\vskip 6ex}
\ncom{\vfour}{\vspace*{8ex}}
\ncom{\norm}{\|\;\;\|}
\ncom{\integ}[4]{\int_{#1}^{#2}\,{#3}\,d{#4}}
\ncom{\vspan}[1]{{{\rm\,span}\{ #1 \}}}
\ncom{\dm}[1]{ {\displaystyle{#1} } }
\ncom{\ri}[1]{{#1} \index{#1}}

\newtheorem{theorem}{\bf Theorem}[section]
\newtheorem{remark}{\bf Remark}[section]
\newtheorem{proposition}{Proposition}[section]

\newtheorem{corollary}{Corollary}[section]

\newtheoremstyle
    {remarkstyle}
    {}
    {11pt}
    {}
    {}
    {\bfseries}
    {:}
    {     }
    {\thmname{#1} \thmnumber{#2} }

\theoremstyle{remarkstyle}

\begin{document}

\newpage

\begin{center}
{\Large \bf Inverse Tempered Stable Subordinators}\\
\end{center}
\vone
\begin{center}
{\bf  A. Kumar$^a$ and P. Vellaisamy$^{a}$}\\
$^{a}${\it Department of Mathematics,
Indian Institute of Technology Bombay,\\ Mumbai-400076, India.}\\

\end{center}

\vtwo
\begin{center}
\noindent{\bf Abstract}
\end{center}
We consider the first-hitting time of a tempered $\beta$-stable subordinator, also called inverse tempered stable (ITS) subordinator. The density function of the ITS subordinator is obtained, for the index of stability $\beta \in (0,1)$. The series representation of the ITS density is also obtained, which could be helpful for computational purposes. The asymptotic behaviors of the $q$-th order moments of the ITS subordinator are investigated. In particular, the limiting behaviors of the mean of the ITS subordinator is given. The limiting form of the ITS density, as the space variable $x\rightarrow 0$, and its $k$-th order derivatives are obtained. The governing PDE for the ITS density is also obtained. The corresponding known results for inverse stable subordinator follow as special cases. 

\vone \noindent{\bf Key words:} Hitting times; inverse
Gaussian process; stable subordinators; tempered stable subordinators.
\vone

\vtwo
\setcounter{equation}{0}

\section{Introduction}
The first-hitting time process (or the first passage time) arises naturally in diverse fields such as finance, insurance, process control and survival analysis (see e.g., Lee and Whitmore (2006)). 
\noindent Let $D(t)$ be a stable process with index of stability $\beta$. The inverse stable process defined by $E(t) = \inf\{s>0:D(s)>t\}$ has been widely used, as a time-change (see, Meerschaert {\it et al.} (2011); Hahn {\it et al.} (2011) and references therein). Tempered stable processes which are useful in several practical applications have also been well studied (see e.g. Rosi\'nski (2007), Meerschaert {\it et al.} (2008b)). Also, inverse tempered stable subordinators are used as a time-change of Brownian motion and Poisson process (see Meerschaert {\it et al.} (2011), Meerschaert {\it et al.} (2013)). 
The closed form expression for the first hitting time density is not easy to obtain for a general stochastic process. However, in case of stable L\'evy process, hitting time density which is also called inverse stable density can be written in terms of stable density itself due to self-similar property of a stable L\'evy process. The focus of this article is on the first hitting times of a tempered stable subordinator, which we call inverse tempered stable (ITS) subordinator. In this article, we have obtained the integral and series representation of the density function of the ITS subordinator. 
Other properties like asymptotic behavior of $q$-th moments of the ITS subordinator are obtained. In particular, mean first-hitting time of the process is discussed in detail, which could be of interest in many applications. Some results concerning the limiting behaviors of the ITS density and its derivatives are obtained. As a special  case, we get the corresponding results for inverse stable processes studied in literature (see e.g. Hahn {\it et al.} (2011); Keyantuo and Lizama (2012)). The series representation of the ITS density is given, which in limiting case as tempering parameter $\lambda\rightarrow 0$ reduces to the series representation of inverse stable density.  

\setcounter{equation}{0}
\section{Inverse tempered stable density}
 Let $\mathcal{L}_t \left(w(x,t)\right) = \int_{0}^{\infty}e^{-st}w(x,t)dt$ denote the Laplace transform (LT) of the function $w$ with respect to the time variable $t$.
 Let $f(x,t)$ denote the density of a $\beta$-stable subordinator $D(t)$. Then the
LT of $f(x,t)$ with respect to the space variable $x$ is 
 \beq\label{Laplace-stable} 
 \mathcal{L}_x (f(x,t))=\E(e^{-sD(t)})=\int_{0}^{\infty}e^{-sx}f(x,t)
dx= e^{-ts^{\beta}}. \eeq
It is well known that all the moments $\E D(t)^{\rho}$ do not exist for $\rho\geq \beta$. To overcome this shortcoming tempered stable distributions are introduced by exponential tempering in the stable distributions (see Rosinski (2007) for more details).
Let $D_{\lambda}(t)$ be the tempered stable subordinator with stability index $\beta$ $(0<\beta<1)$ and the tempering parameter $\lambda >0$.
A tempered stable subordinator
$D_{\lambda}(t)$ with index $\beta$ has the density

\beq\label{ts-density} 
f_{\lambda}(x,t)= e^{-\lambda x+\lambda^{\beta}t} f(x,t),~~ \lambda>0, 
\eeq
which has all the moments finite and is also infinitely divisible, but not self-similar. Further, the LT of $f_{\lambda}(x,t)$ is
\beq\label{tempered-LT}
\mathcal{L}_x\Big({f}_{\lambda}(x,t)\Big)=\int_{0}^{\infty}e^{-sx}f_{\lambda}(x,t)dx =
e^{-t\big((s+\lambda)^{\beta}-\lambda^{\beta}\big)},
 \eeq
see Meerschaert et al. (2013). 
 Let
 $E_{\lambda}(t)$ be the right continuous inverse of $D_{\lambda}(t)$, defined by
\bea
E_{\lambda}(t) = \inf\{u>0: D_{\lambda}(u)>t\},~~ t\geq 0.
\eea  
For a non-decreasing L\'evy process $D(u)$ with corresponding L\'evy measure $\pi_D$ and density function $f$, we have (see e.g. Bertoin (1996); Sato (1999)) from L\'evy-Khinchin representation
 \bea
 \int_{0}^{\infty}e^{-st}f_{D(x)}(t)dt = e^{-x\Psi_D(s)},
 \eea
 where 
 \beq\label{lsym}
 \Psi_D(s) = bs+\int_{0}^{\infty}(1-e^{-su})\pi_D(du),~~s>0,
 \eeq
 is the Laplace symbol. 
The
L\'evy measure density  corresponding to a tempered stable process is given
by (see e.g. Cont and Tankov, 2004, p. 115)
\bea \pi_{D_\lambda}(x) = \frac{ce^{-\lambda
x}}{x^{\beta+1}},~c>0, ~x>0, \eea which implies $\displaystyle
\int_{0}^{\infty}\pi_{D_\lambda}(x)dx = \infty$ and hence using
Theorem 21.3 of Sato (1999),
  the sample paths of $D_{\lambda}(t)$ are strictly increasing, since jumping times are dense in $(0,\infty)$. 
Since the sample paths of $D_{\lambda}(t)$ are strictly increasing with jumps, the sample paths of $E_{\lambda}(t)$ are almost surely continuous and are constant over the intervals where $D_{\lambda}(t)$ have jumps.
Further, the relation
\beq \label{relationD-E}
\{E_{\lambda}(t)\leq x\} = \{D_{\lambda}(x)\geq t\},
\eeq  
holds.
\noindent For a strictly increasing subordinator $Y(t)$ with density function $p(x,t)$ and Laplace symbol $\Psi_Y(s)$, the density function $q(x,t)$ of the hitting time process has the LT (e.g., see Meerschaert and Scheffler (2008a)) 
\beq\label{LT-of-inverse-subordinator}
\mathcal{L}_t(q(x,t)) = \frac{1}{s}\Psi_Y(s)e^{-x\Psi_Y(s)}.
\eeq
Since $D_{\lambda}(t)$ is strictly increasing subordinator with Laplace symbol $\Psi_{D_{\lambda}}(s) = (s+\lambda)^{\beta}-\lambda^{\beta}$, we obtain from \eqref{LT-of-inverse-subordinator} the LT of $h_{\lambda}(x,t)$ with respect to the time variable as
 \beq\label{ch3-prop2.1}
\mathcal{L}_t(h_{\lambda}(x,t)) = \frac{1}{s}\big((s+\lambda)^{\beta}-\lambda^{\beta}\big)e^{-x\big((s+\lambda)^{\beta}-\lambda^{\beta}\big)}.
\eeq
We first invert the Laplace transform of the density of the process $E_{\lambda}(t)$ with respect to the time variable to
get the corresponding density function $h_{\lambda}(x,t)$ in explicit form. 
\begin{theorem}{\rm 
 The density function $h_{\lambda}(x,t)$ of $E_{\lambda}(t)$ admits following integral form
 \beq\label{density-its}
 h_{\lambda}(x,t) = \frac{1}{\pi} e^{\lambda^{\beta} x -\lambda t}\int_{0}^{\infty}\frac{e^{-ty-xy^{\beta}\cos (\beta\pi)}}{y+\lambda}
 \left[\lambda^{\beta} \sin(xy^{\beta} \sin(\beta\pi)) + y^{\beta}\sin(\beta\pi- x y^{\beta}\sin(\beta\pi))\right] dy,
 \eeq
where $x>0$, $\lambda>0$ and $0<\beta <1.$                 }
\end{theorem}
\noindent \begin{proof} 
Let $F(x,s)= \mathcal{L}_t(h_{\lambda}(x,t))$. Then from \eqref{ch3-prop2.1}
\begin{align} \label{laplace}
F(x,s) &= e^{\lambda^{\beta}x}\left(\frac{(s+\lambda)^{\beta} - \lambda^{\beta}}{s}\right)e^{-x(s+\lambda)^{\beta}}.
\end{align}
The density function of $h_{\lambda}(x,t)$ can be obtained by using the Laplace 
inversion formula, namely,
\beq\label{complex-inversion}
h_{\lambda}(x,t) = \frac{1}{2\pi i} \int_{x_0-i\infty}^{x_0+i\infty}e^{st}F(x,s)ds,
\eeq
(see Schiff (1999), p. 152).
\begin{figure}[ht]
\centering{\includegraphics[scale=0.75]{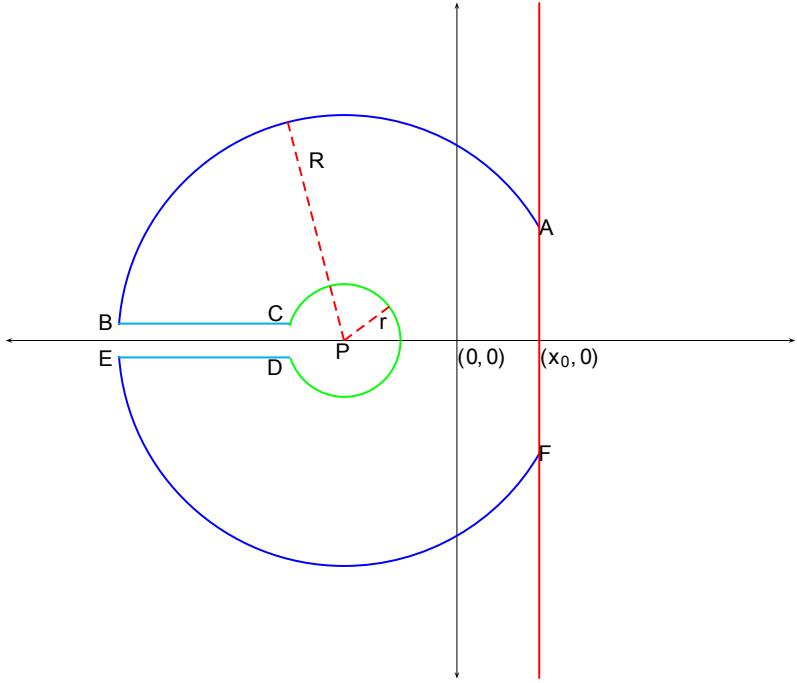}}
\caption{Contour ABCDEFA}
\end{figure}
\noindent For calculating integral in \eqref{complex-inversion}, we consider a closed key-hole contour $\mathcal{C}: ABCDEFA$ (see Fig. 1) with a branch point at P$=(-\lambda, 0)$.
Here AB and EF are arcs of a circle of radius $R$ with center at $P$, BC and DE are line segments parallel to $x$-axis as shown in the Figure 1, CD is an arc $\gamma_r$ of a circle of radius $r$ with center at P and FA is the line segment from $x_0-iy$ to $x_0+iy$ with $x_0>0$. 
By residue theorem, we have
\begin{equation}\label{contour-1}
 \begin{split}
\frac{1}{2\pi i}\int_{\mathcal{C}}e^{st}F(x,s) ds &= \sum  Res \, F(x,s)\\
&= 0,
\end{split}
\end{equation} 
since the residue of $F$ at simple pole $s=0$, is zero. 

\noindent It is easy to see that
\beq\label{ch3-limit}
\lim_{R\rightarrow\infty}\int_{AB}e^{st}F(x,s) ds =  - \lim_{R\rightarrow\infty}\int_{EF}e^{st}F(x,s) ds.
\eeq
Putting $s = -\lambda + re^{i\theta}$ on the arc CD, we have 
\begin{align}\label{path-cd}
\Big|\int_{CD}e^{st}F(x,s) ds\Big| &= e^{\lambda^{\beta} x}\Big|\int_{CD}\frac{e^{st}}{s}((s+\lambda)^{\beta} - \lambda ^{\beta}) e^{- x (s+\lambda)^{\beta}}ds\Big|\nonumber\\
&\leq e^{\lambda^{\beta} x}\int_{\pi-\epsilon}^{-\pi+\epsilon}\Big|\frac{e^{t(-\lambda+re^{i\theta})}}{-\lambda+re^{i\theta}}(r^{\beta}e^{i\beta\theta}-\lambda^{\beta})e^{- x r^{\beta}e^{i\beta\theta}}
(i r e^{i\theta})\Big|d\theta \nonumber\\
&\leq re^{\lambda^{\beta} x}\int_{\pi-\epsilon}^{-\pi+\epsilon}\frac{e^{tr\cos\theta-xr^{\beta}\cos \beta\theta}}{|-\lambda+re^{i\theta}|} (r^{\beta} + \lambda^{\beta}) d\theta   \nonumber\\
&\leq \frac{re^{\lambda^{\beta} x}}{|(|\lambda|-|r|)|}\int_{\pi-\epsilon}^{-\pi+\epsilon}{e^{tr\cos\theta-xr^{\beta}\cos \beta\theta}} (r^{\beta} + \lambda^{\beta})d\theta  \nonumber \\
&\longrightarrow 0, 
\end{align} 
as $r\rightarrow 0$, since the integrand is bounded.
\noindent Along BC, put $s+\lambda= ye^{i \pi}$ so that $(s+\lambda)^{\beta} =y^{\beta}e^{i\beta\pi}$ and $ds =-dy$. We have
\begin{equation}\label{path-bc}\begin{split}
 \int_{BC}e^{st}F(x,s) ds &= e^{\lambda^{\beta} x}\int_{-R-\lambda}^{-r-\lambda}\frac{e^{st}}{s}((s+\lambda)^{\beta}-\lambda^{\beta})e^{- x (s+\lambda)^{\beta}}ds\\
 &= e^{\lambda^{\beta} x -t\lambda}\int_{R}^{r}\frac{e^{-ty}}{y+\lambda}(y^{\beta}e^{i\beta\pi}-\lambda^{\beta})e^{-xy^{\beta}e^{i\beta\pi}}dy.
 \end{split}
\end{equation}
\noindent Next along DE, put $s+\lambda= ye^{-i \pi}$ so that $(s+\lambda)^{\beta} =y^{\beta}e^{-i\beta\pi}$ and $ds =-dy$. Hence,
\begin{equation}\label{path-de}\begin{split}
 \int_{BC}e^{st}F(x,s) ds &= e^{\lambda^{\beta} x}\int_{-r-\lambda}^{-R-\lambda}\frac{e^{st}}{s}((s+\lambda)^{\beta}-\lambda^{\beta})e^{- x (s+\lambda)^{\beta}}ds\\
 &= e^{\lambda^{\beta} x - t\lambda}\int_{r}^{R}\frac{e^{-ty}}{y+\lambda}(y^{\beta}e^{-i\beta\pi}-\lambda^{\beta})e^{-xy^{\beta}e^{-i\beta\pi}}dy.
 \end{split}
\end{equation}
Using \eqref{path-bc} and \eqref{path-de}, we get 
 \begin{align}\label{sum-bc-de}
\frac{1}{2\pi i}&\int_{BC}e^{st}F(x,s) ds + \frac{1}{2\pi i}\int_{DE}e^{st}F(x,s)ds \nonumber\\
&=- \frac{e^{\lambda^{\beta} x - t\lambda}}{\pi}\int_{r}^{R}\frac{e^{-ty-xy^{\beta}\cos(\beta\pi)}}{\lambda+y}\Big[ y^{\beta}\sin(\beta\pi-xy^{\beta}\sin{(\beta\pi)}) 
+ \lambda^{\beta} \sin (xy^{\beta} \sin (\beta\pi)) \Big]dy.
  \end{align}
Using \eqref{contour-1} -- \eqref{path-cd} and \eqref{sum-bc-de} with $r\rightarrow 0, R\rightarrow\infty$, we get
\begin{align}\label{f1}
\frac{1}{2\pi i} &\int_{x_0-i\infty}^{x_0+i\infty}e^{st}F(x,s)ds \nonumber\\
&= \frac{e^{\lambda^{\beta} x - t\lambda}}{\pi}\int_{0}^{\infty}\frac{e^{-ty-xy^{\beta}\cos(\beta\pi)}}{\lambda+y}\Big[ y^{\beta}\sin(\beta\pi-xy^{\beta}\sin{(\beta\pi)}) 
+ \lambda^{\beta} \sin (xy^{\beta} \sin (\beta\pi)) \Big]dy.
\end{align}

\noindent The result follows now by using \eqref{laplace} and \eqref{f1} with \eqref{complex-inversion}.
 \end{proof} 

\begin{remark} {\rm
Using a contour similar to Figure 1 with branch point at origin and using similar arguments, we can obtain the density $f(x,t)$ of a stable subordinator as 
\begin{equation}\label{stable-den-int}
f(x,t) = \frac{1}{\pi}\int_{0}^{\infty} e^{-ux}e^{-tu^{\beta}\cos \beta\pi} \sin(tu^{\beta}\sin \beta\pi) du.
\end{equation}
\noindent Now, using \eqref{stable-den-int}, the density function of ITS subordinator can also be obtained as follows
\begin{align*}
\P(E_{\lambda}(t) \leq x) &= \P(D_{\lambda}(x) \geq t) = \int_{t}^{\infty}f_{\lambda}(v,x)dv\\
& = \int_{t}^{\infty} e^{-\lambda v+ \lambda^{\beta}x}f(v,x)dv ~~\text{(using \eqref{ts-density})}\\
& = \frac{e^{\lambda^{\beta}x}}{\pi}\int_{t}^{\infty}\int_{0}^{\infty} e^{-v(\lambda + u)} e^{-xu^{\beta}\cos\beta\pi}\sin(xu^{\beta}\sin\beta\pi)du dv ~\text{(using \eqref{stable-den-int})}\\
& = \frac{e^{\lambda^{\beta}x}}{\pi}\int_{0}^{\infty} \frac{e^{-t(\lambda + u)}}{\lambda +u} e^{-xu^{\beta}\cos\beta\pi}\sin(xu^{\beta}\sin\beta\pi)du.
\end{align*}
Theorem 2.1 now follows by taking the derivative of both sides with respect to $x$. }
\end{remark}

\begin{remark} {\rm
When $\beta=1/2$, we have from \eqref{density-its}
\begin{equation}
 h_{\lambda}(x,t)=\frac{e^{\sqrt{\lambda}x-\lambda t}}{\pi}\int_{0}^{\infty}\frac{e^{-ty}}{y+\lambda}\left(\sqrt{\lambda}\sin(x\sqrt{y})+\sqrt{y}\cos(x\sqrt{y})\right)dy,
\end{equation}
which is the density function of hitting time of inverse Gaussian process (see Vellaisamy and Kumar (2013)), as expected. }
\end{remark}

\noindent Using NIntegrate and Plot functions of Mathematica 8.0, we plot the densities functions of ITS subordinator for $\beta \in \{0.2, 0.4, 0.6 \}$ and $\lambda = t =1$. The densities become more peaked for increasing values of $\beta$ (see Figure 2).
 \begin{figure}[ht]
\centering{\includegraphics{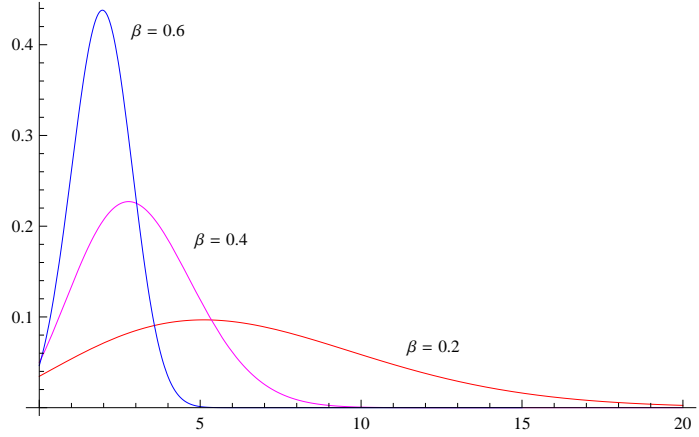}}
\caption{Density functions of ITS Subordinators}
\end{figure}

\begin{remark} {\rm
When $\beta=1/2$, we have from \eqref{density-its}
\begin{equation}
 h_{\lambda}(x,t)=\frac{e^{\sqrt{\lambda}x-\lambda t}}{\pi}\int_{0}^{\infty}\frac{e^{-ty}}{y+\lambda}\left(\sqrt{\lambda}\sin(x\sqrt{y})+\sqrt{y}\cos(x\sqrt{y})\right)dy,
\end{equation}
which is the density function of hitting time of inverse Gaussian process (see Vellaisamy and Kumar (2013)), as expected. }
\end{remark}

\noindent Let $\Gamma(a,u)$ be the incomplete gamma function defined by
\beq\label{gamma-incomplete}
\Gamma(a,u) = \int_{u}^{\infty}y^{a-1}e^{-y}dy,
\eeq
where $u>0$ and $a\in\mathbb{R}$. Note that the integral in \eqref{gamma-incomplete} exists and is real.
The following integral will be used further in computations. Let $p$ and $q$ be positive. Then
\begin{align}\label{incomplete-gamma}
 \int_{0}^{\infty}\frac{e^{-ty}y^p}{(y+q)}dy&=\int_{0}^{\infty}e^{-ty}y^p\left(\int_{0}^{\infty}e^{-(y+q)u}du\right)dy\nonumber\\
 &=\int_{0}^{\infty}e^{-qu}\left(\int_{0}^{\infty}y^pe^{-(t+u)y}dy\right)du\nonumber\\
 &=\Gamma(p+1)\int_{0}^{\infty}\frac{e^{-qu}}{(t+u)^{p+1}}du\nonumber\\
 &=\Gamma(p+1)q^pe^{qt}\int_{qt}^{\infty}w^{-p-1}e^{-w}dw\nonumber\\
 &=\Gamma(p+1)q^pe^{qt}\Gamma(-p,qt)~~(\mbox{using}~\eqref{gamma-incomplete}).
\end{align}
The following series representation is useful for numerical computational purposes.

\begin{proposition}\label{series-ITS} {\rm
The series representation of the density $h_{\lambda}(x,t)$ of $E_{\lambda}(t)$ is given by
\begin{align}
h_{\lambda}(x,t) &= \frac{e^{\lambda^{\beta}x}}{\pi}\sum_{k=0}^{\infty}(-1)^k\frac{x^k\lambda^{\beta(k+1)}}{k!}\Big[\Gamma(1+\beta(k+1))\Gamma(-\beta(k+1),\lambda t)\sin((k+1)\beta\pi) \nonumber\\
&\hspace{3cm}-\Gamma(1+\beta k)\Gamma(-\beta k,\lambda t)\sin (k\beta\pi)\Big],
\end{align} 
where $x>0$, $\lambda>0$ and $0<\beta <1.$   }
\end{proposition}

\begin{proof}
 Note from \eqref{density-its}
 \begin{align*}
  h_{\lambda}(x,t) &=\frac{e^{\lambda^{\beta}x-\lambda t}}{\pi}Im\Big[\int_{0}^{\infty}\frac{e^{-ty-xy^{\beta}\cos \beta\pi}}{y+\lambda}\lambda^{\beta}e^{ixy^{\beta}\sin \beta\pi}dy\\
  &\hspace{3.5cm}+ \int_{0}^{\infty}\frac{e^{-ty-xy^{\beta}\cos \beta\pi}}{y+\lambda}y^{\beta}e^{i\beta\pi-ixy^{\beta}\sin \beta\pi}dy\Big]\\
  &=\frac{e^{\lambda^{\beta}x-\lambda t}}{\pi}Im\Big[\lambda^{\beta}\int_{0}^{\infty}\frac{e^{-ty-xy^{\beta}e^{-i\beta\pi}}}{y+\lambda}dy
 + e^{i\beta\pi}\int_{0}^{\infty}\frac{e^{-ty-xy^{\beta}e^{i\beta\pi}}}{y+\lambda}y^{\beta}dy\Big]\\
&=\frac{e^{\lambda^{\beta}x-\lambda t}}{\pi}Im\Big[\lambda^{\beta}\sum_{k=0}^{\infty}(-1)^k\frac{x^ke^{-ik\beta\pi}}{k!}\int_{0}^{\infty}\frac{e^{-ty}y^{\beta k}}{y+\lambda}dy\\
 &\hspace{2.5cm}+ \sum_{k=0}^{\infty}(-1)^k\frac{x^ke^{i(k+1)\beta\pi}}{k!}\int_{0}^{\infty}\frac{e^{-ty}y^{\beta (k+1)}}{y+\lambda}dy\Big].
 \end{align*}

\noindent Further, using the relationship in \eqref{incomplete-gamma}, we have 
 \begin{align*}
h_{\lambda}(x,t) &=\frac{e^{\lambda^{\beta}x-\lambda t}}{\pi}Im\Big[\lambda^{\beta}\sum_{k=0}^{\infty}(-1)^k\frac{x^ke^{-ik\beta\pi}}{k!}e^{\lambda t}\lambda^{\beta k} \Gamma(1+\beta k)\Gamma(-\beta k, \lambda t)\\
 &\hspace{1cm}+ \sum_{k=0}^{\infty}(-1)^k\frac{x^ke^{i(k+1)\beta\pi}}{k!}e^{\lambda t}\lambda^{\beta (k+1)} \Gamma(1+\beta (k+1))\Gamma(-\beta (k+1), \lambda t)\Big]\\
 &=\frac{e^{\lambda^{\beta}x}}{\pi}\sum_{k=0}^{\infty}(-1)^k\frac{x^k\lambda^{\beta(k+1)}}{k!}\Big[\Gamma(1+\beta(k+1))\Gamma(-\beta(k+1),\lambda t)\sin(k+1)\beta\pi \nonumber\\
&\hspace{3cm}-\Gamma(1+\beta k)\Gamma(-\beta k,\lambda t)\sin (k\beta\pi)\Big],
\end{align*}
and hence the result.
\end{proof}

\noindent Let $f(x,1)$ be the density
function of a  $\beta$-stable $(0<\beta< 1)$ random variable
$D(1)$ with LT $e^{-s^{\beta}}$. It is well known (e.g., see Feller (1971), p. 583; Uchaikin and Zolotarev (1999), p. 106), that
\beq \label{stable-density}
f(x,1) = \frac{1}{\pi}\sum_{k=0}^{\infty}\frac{\Gamma(k\beta+1)}{k!}(-1)^{k+1}x^{-\beta k-1}\sin(k\beta\pi).
\eeq
Let $E(t) = \inf\{s>0: D(s)>t\}$ be the right continuous inverse of $D(t)$. Then,
\begin{align*}
P(E(t)\leq x) = P(D(x)\geq t)=P(x^{1/\beta}D(1)\geq t)=P(D(1)\geq tx^{-1/\beta}).
\end{align*}
This implies 
\beq \label{inverse-stable-density} 
f_{E(t)}(x) =
\frac{t}{\beta}x^{-1-1/{\beta}}f(tx^{-1/{\beta}},1), ~~x>0,
\eeq 
which is the density function of the hitting time process $E(t)$, also called inverse stable subordinator.
Using \eqref{stable-density} and \eqref{inverse-stable-density}, we have
\beq\label{series-inverse-stable}
f_{E(t)}(x) = \frac{1}{\pi}\sum_{k=1}^{\infty}(-1)^{k-1}\frac{\Gamma(k\beta)}{(k-1)!}t^{-\beta k}x^{k-1}\sin(k\beta\pi).
\eeq
Also, putting $\lambda=0$ in \eqref{density-its}, we get the integral representation for $f_{E(t)}(x)$ as
\begin{equation}\label{is-density}
h_0(x,t) = \frac{1}{\pi}\int_{0}^{\infty}e^{-ty-xy^{\beta}\cos (\beta\pi)}
y^{\beta-1}\sin(\beta\pi- x y^{\beta}\sin(\beta\pi)) dy,
\end{equation}
which corresponds to the hitting time densities of a stable process. The next result shows that \eqref{series-inverse-stable} and \eqref{is-density} are the same.
\begin{proposition}\label{series-IS}
Equation \ref{is-density} is indeed the integral representation of inverse stable density.
\end{proposition}
\begin{proof}
Using \eqref{is-density}, we have
\begin{align*}
h_0(x,t)&= Im\left[\frac{1}{\pi}\int_{0}^{\infty}e^{-ty-xy^{\beta}\cos (\beta\pi)}
y^{\beta-1} e^{i(\beta\pi-xy^{\beta}\sin \beta\pi)}dy\right]\\
&= Im\left[\frac{e^{i\beta\pi}}{\pi}\int_{0}^{\infty}e^{-ty} y^{\beta-1}e^{-xy^{\beta}e^{i\beta\pi}}dy\right]\\
&= Im\left[\frac{e^{i\beta\pi}}{\pi}\int_{0}^{\infty}e^{-ty} y^{\beta-1}\sum_{k=0}^{\infty}(-1)^k\frac{x^ky^{\beta k}e^{ik\beta\pi}}{k!}dy\right]\\
&= Im\left[\frac{1}{\pi}\sum_{k=0}^{\infty}(-1)^k\frac{\Gamma((k+1)\beta)}{k!}x^kt^{-(k+1)\beta}e^{i(k+1)\beta\pi}\right]\\
&=\frac{1}{\pi}\sum_{k=0}^{\infty}(-1)^k\frac{\Gamma((k+1)\beta)}{k!}x^kt^{-(k+1)\beta}\sin((k+1)\beta\pi)\\
&=\frac{1}{\pi}\sum_{k=1}^{\infty}(-1)^{k-1}\frac{\Gamma(k\beta)}{(k-1)!}t^{-\beta k}x^{k-1}\sin(k\beta\pi),
 \end{align*}
which coincides with \eqref{series-inverse-stable}.
\end{proof}
\noindent It is well known that for $a<0$, as $z\rightarrow 0,$ (e.g., see Abramowitz and Stegun (1992))
\beq\label{gamma-asym}
\frac{\Gamma(a,z)}{z^a}\rightarrow -\frac{1}{a}.
\eeq

\begin{remark} {\rm
The result in Proposition \ref{series-IS} can also be derived by using \eqref{gamma-asym} and Proposition \ref{series-ITS}. Using \eqref{gamma-asym}, we have $\Gamma(-\beta k, \lambda t)/\lambda^{-\beta k}\rightarrow t^{-\beta k}/\beta k,$ as $\lambda\rightarrow 0$, and hence
\begin{align*}
\lim_{\lambda\rightarrow 0}h_{\lambda}(x,t) &= \frac{1}{\pi}\sum_{k=0}^{\infty}(-1)^k\frac{\Gamma(1+(k+1)\beta)}{k!}\frac{\sin((k+1)\beta\pi)}{\beta(k+1)} x^kt^{-(k+1)\beta}\\
&=\frac{1}{\pi}\sum_{k=1}^{\infty}(-1)^{k-1}\frac{\Gamma(k\beta)}{(k-1)!}t^{-\beta k}x^{k-1}\sin(k\beta\pi).
\end{align*}

}
\end{remark}

\setcounter{equation}{0}
\section{Asymptotic behavior of moments}

\noindent It looks difficult to obtain the explicit expressions for the moments of the ITS subordinator, for an arbitrary $\beta\in(0,1).$ However, the asymptotic behavior of the first moment that is also called the mean first-hitting time is of much interest. First, we obtain the LT of the $q$-th raw moment of $E_{\lambda}(t)$. For $q>0$, let $M_q(t)=\E(E_{\lambda}^q(t))$. Then
\begin{align}
\tilde{M}_q(s)&= \int_{0}^{\infty}e^{-st}M_q(t)dt= -\int_{0}^{\infty}e^{-st}\left(\int_{0}^{\infty}y^{q}\frac{d}{dy}P(E_{\lambda}(t)> y)dy \right)dt\nonumber\\
 &=q\int_{0}^{\infty}e^{-st}\left(\int_{0}^{\infty}y^{q-1}P(E_{\lambda}(t)>y)dy\right) dt\nonumber\\
 &=q\int_{0}^{\infty}y^{q-1}\left(\int_{0}^{\infty}e^{-st}P(D_{\lambda}(y)\leq t)dt\right) dy ~~~\mbox{(Using \eqref{relationD-E})}\nonumber\\
 & = \frac{q}{s}\int_{0}^{\infty}y^{q-1}\left(\int_{0}^{\infty}e^{-st}f_{D_{\lambda}(y)}(t)dt\right) dy\nonumber\\
&=\frac{q}{s}\int_{0}^{\infty}y^{q-1}e^{-y\Psi_{D_{\lambda}}(s)}dy\nonumber\\
&=\frac{\Gamma(1+q)}{s\Psi_{D_{\lambda}}(s)^q},
\end{align}
where $\Psi_{D_{\lambda}}(s)=(s+\lambda)^{\beta}-\lambda^{\beta}.$ \\

\noindent Veillette and Taqqu (2010) obtained similar expression for the case $q=1.$ For asymptotic behavior of $M_q(t)$, we use Tauberian theorem. First we recall that a function $L(t)$ is {\it slowly varying} at some $t_0$, if for all fixed $c>0$, $\lim_{t\rightarrow t_0}{L(ct)}/{L(t)}=1$. For readers convenience, we state here the Tauberian theorem (see Bertoin (1996), p. 10).

\begin{theorem}[Tauberian Theorem]\label{Taubarian} Let $L:(0,\infty)\rightarrow(0,\infty)$ be a slowly varying function at $0$ (respectively $\infty$) and let $\alpha\geq 0.$ Then for a function $M: (0,\infty)\rightarrow(0,\infty)$, the following are equivalent: 
$$(i)~~~~~ M(t)\sim t^{\alpha}L(t)/\Gamma(1+\alpha),~~  t\rightarrow 0 ~(\mbox{respectively}~ t\rightarrow\infty).$$
$$\hspace{-.2cm}(ii)~~~~~~ \tilde{M}(s)\sim s^{-\alpha-1}L(1/s), ~~ s\rightarrow \infty~ (\mbox{respectively}~ s\rightarrow 0),$$
where $f(x)\sim g(x)$ as $x\rightarrow x_0$ means that $\lim_{x\rightarrow x_0}f(x)/g(x)=1.$
\end{theorem}

\begin{proposition}
 The $q$-th moment of $E_{\lambda}(t)$ satisfies
 \begin{eqnarray*}
 M_q(t) \sim \left\{
\begin{array}{ll}
\displaystyle \frac{\Gamma(1+q)}{\Gamma(1+q\beta)}t^{q\beta},~~~\mbox{as}~ t\rightarrow 0,\\
\frac{\lambda^{q(1-\beta)}\Gamma(1+q)}{\beta^q}t^q, ~~~\mbox{as}~t\rightarrow\infty.
\end{array}\right.
\end{eqnarray*}
\end{proposition}
\begin{proof}
 We have, as $s\rightarrow\infty,$
 \begin{align*}
  \tilde{M}_q(s) = \frac{\Gamma(1+q)}{s\left((s+\lambda)^{\beta}-\lambda^{\beta}\right)^q}
  \sim \frac{\Gamma(1+q)}{s^{1+q\beta}}.
 \end{align*}
Using Theorem \ref{Taubarian}, we have $M_q(t)\sim \frac{\Gamma(1+q)}{\Gamma(1+q\beta)}t^{q\beta}$, as $t\rightarrow 0$. 

\noindent Similarly, we have 
\bea
\tilde{M}_q(s) \sim \frac{\lambda^{q(1-\beta)}\Gamma(1+q)}{\beta^q} s^{-q-1}, ~~\mbox{as}~ s\rightarrow 0,
\eea
using again Theorem \ref{Taubarian}.
\end{proof}
\begin{remark}(i)
 The mean hitting time $M_1(t)$ has the the following asymptotic behaviors (see Stanislavsky {\it et al.} (2008))
 \begin{eqnarray*}
 M_1(t)=\E (E_{\lambda}(t)) \sim \left\{
\begin{array}{ll}
\displaystyle \frac{t^{\beta}}{\Gamma(1+\beta)},~~~~~\mbox{as}~ t\rightarrow 0,\\
\frac{\lambda^{1-\beta}}{\beta}t, ~~~~~~~~\mbox{as}~t\rightarrow\infty.
\end{array}\right.
\end{eqnarray*} 

\noindent (ii) For a L\'evy process $Z(t)$ with finite mean, we have $\E Z(t)=t\E Z(1),$ $\forall ~t>0.$ Here, $M_1(t)\sim t^{\beta}/\Gamma(1+\beta),$ as $t\rightarrow 0$. Hence, $E_{\lambda}(t)$ is not a L\'evy process.
\end{remark}

\setcounter{equation}{0}
\section{Further properties of $h_{\lambda}(x,t)$}
In this section, we study some additional properties and the pde's associated with $h_{\lambda}(x,t).$
\begin{proposition}
 The density function $h_{\lambda}(x,t)$ have following interesting properties:\\
 \noindent(a)  For $0<\beta <1,$
 \begin{align}\label{prop2.1a}
 \lim_{x\rightarrow 0^{+}} h_{\lambda}(x,t) = 
 \frac{\sin \beta\pi}{\pi}\lambda^{\beta} \Gamma(1+\beta)\Gamma(-\beta, \lambda t).
 \end{align}
 (b) For $0<\beta <1,$
 \begin{align}\label{prop2.1b} 
 &\frac{d^k}{dx^k}h_{\lambda}(x,t)\Big|_{x=0} \nonumber\\
&=\frac{e^{-\lambda t}}{\pi}\int_{0}^{\infty}\frac{e^{-ty}}{y+\lambda}\left[(\lambda^{2\beta}+y^{2\beta}-2\lambda^{\beta}y^{\beta}\cos(\beta\pi))^{k/2}(\lambda^{\beta}\sin(k\alpha)+y^{\beta}\sin(k\alpha-\beta\pi))\right]dy,
 \end{align}
 where $\tan \alpha = y^{\beta}\sin(\beta\pi)/(\lambda^{\beta}-y^{\beta}\cos(\beta\pi)).$\\
 
\noindent (c) For $\beta=1/m,$ $m\geq 2$, the density function $h_{\lambda}(x,t)$ satisfies 
 \beq\label{prop2.1c}
 \sum_{j=1}^{m}(-1)^j{\binom{m}{j}}\lambda^{(1-\frac{j}{m})}\frac{\partial^j}{\partial x^j}h_{\lambda}(x,t) = \frac{\partial}{\partial t}h_{\lambda}(x,t)+h_{\lambda}(x,0)\delta_0(t).
 \eeq
\end{proposition}
\begin{proof}
\noindent (a) Let
\begin{equation*}
 I(x,t,y) = \frac{e^{-ty-xy^{\beta}\cos (\beta\pi)}}{y+\lambda}
 \left[\lambda^{\beta} \sin(xy^{\beta} \sin(\beta\pi)) + y^{\beta}\sin(\beta\pi- x y^{\beta}\sin(\beta\pi))\right].
\end{equation*}
 Then, 
 \begin{align*}
 |I(x,t,y)|\leq \frac{e^{-ty-xy^{\beta}\cos\beta\pi}}{y+\lambda}(\lambda^{\beta}+y^{\beta}). 
 \end{align*}
 Note that $\cos\beta\pi$ is positive or negative depending on the value of $\beta$. Hence,
 \begin{align*}
  \int_{0}^{\infty}|I(x,t,y)|dy &\leq \int_{0}^{\lambda}\frac{e^{-ty-xy^{\beta}\cos\beta\pi}}{y+\lambda}(\lambda^{\beta}+y^{\beta})dy\\
  & \hspace{1cm}+ \int_{\lambda}^{\infty}\frac{e^{-ty-xy^{\beta}\cos\beta\pi}}{y+\lambda}(\lambda^{\beta}+y^{\beta})dy.
 \end{align*}
 The first integral is finite since integrand is bounded and limits of integration are finite. For the second integral, we have
 
  \begin{align}\label{integrability}
   \int_{\lambda}^{\infty}\frac{e^{-ty-xy^{\beta}\cos\beta\pi}}{y+\lambda}(\lambda^{\beta}+y^{\beta})dy &\leq \frac{1}{\lambda}\int_{\lambda}^{\infty}{e^{-ty-xy^{\beta}\cos\beta\pi}}y^{\beta}dy\nonumber\\
   &\leq \frac{1}{\lambda}\int_{\lambda}^{\infty}e^{-ty+xy^{\beta}}y^{\beta}dy\nonumber\\
   &= \frac{1}{\lambda(1+\beta)}\int_{\lambda^{1+\beta}}^{\infty}e^{-tu^{1/(1+\beta)}+xu^{\beta/(1+\beta)}}du~~~~~(\mbox{put}~~y^{1+\beta}=u)\nonumber\\
   &=\frac{1}{\lambda(1+\beta)}\int_{\lambda^{1+\beta}}^{\infty}e^{-tu^{1/(1+\beta)}\big(1-\frac{x}{t}\frac{1}{u^{(1-\beta)/(1+\beta)}}\big)}du\nonumber\\
   &\leq \frac{1}{\lambda(1+\beta)}\int_{\lambda^{1+\beta}}^{\infty}e^{-tu^{1/(1+\beta)}}du ~~~(\mbox {for sufficiently large u})\nonumber\\
  &<\infty.
  \end{align}
  Now using dominated convergence Theorem (DCT), we have
 \begin{equation}\label{eq4.5}
  \begin{split}
  \lim_{x\rightarrow 0^{+}}h_{\lambda}(x,t) &= \frac{e^{-\lambda t}}{\pi}\int_{0}^{\infty}I(0,t,y)dy\\
  &=\frac{e^{-\lambda t}}{\pi}\sin(\beta \pi)\int_{0}^{\infty}\frac{y^{\beta}e^{-ty}}{y+\lambda}dy\\
  &=\frac{\sin \beta\pi}{\pi}\lambda^{\beta} \Gamma(1+\beta)\Gamma(-\beta, \lambda t) ~~(\mbox{using} \eqref{incomplete-gamma}).
 \end{split}
 \end{equation}
 
 \noindent (b) We have 
 \begin{align*}
 &\frac{\partial}{\partial x}I(x,t,y) =  \frac{e^{-ty-xy^{\beta}\cos (\beta\pi)}}{y+\lambda}(-y^{\beta}\cos(\beta\pi))
 \left[\lambda^{\beta} \sin(xy^{\beta} \sin(\beta\pi)) + y^{\beta}\sin(\beta\pi- x y^{\beta}\sin(\beta\pi))\right]\\
 &\hspace{1cm}+\frac{e^{-ty-xy^{\beta}\cos (\beta\pi)}}{y+\lambda}\left[\lambda^{\beta}y^{\beta}\sin(\beta\pi) \cos(xy^{\beta}\sin(\beta\pi))-y^{2\beta}\sin(\beta\pi)\cos(\beta\pi-xy^{\beta}\sin\beta\pi)\right].
 \end{align*}
 
\noindent This implies 
 \begin{equation*}
  \Big|\frac{\partial}{\partial x}I(x,t,y)\Big|\leq 2\frac{e^{-ty-xy^{\beta}\cos\beta\pi}}{y+\lambda}(\lambda^{\beta}y^{\beta}+y^{2\beta}),
 \end{equation*}
 which is independent of $x$ and integrable similar to \eqref{integrability}. Thus,
 \beq
 \frac{\partial}{\partial x}h_{\lambda}(x,t) = \frac{e^{\lambda^{\beta}x}-\lambda t}{\pi}\left[\int_{0}^{\infty}\frac{\partial}{\partial x}I(x,t,y)dy+\lambda^{\beta}\int_{0}^{\infty}I(x,t,y)dy\right]
 \eeq
 Similarly, we can show that $h_{\lambda}(x,t)$ is infinitely differentiable. 
 We have
 \begin{align*}
  \frac{d^k}{dx^k}h_{\lambda}(x,t) &= \frac{e^{-\lambda t}}{\pi}\int_{0}^{\infty}\frac{e^{-ty}}{y+\lambda}\frac{d^k}{dx^k}\Big[e^{\lambda^{\beta}x-xy^{\beta}\cos(\beta\pi)}\Big(\lambda^{\beta}\sin(xy^{\beta}\sin(\beta\pi))\\
  &\hspace{1cm}+y^{\beta}\sin(\beta\pi-xy^{\beta}\sin(\beta\pi)\Big)dy\Big].
 \end{align*}
Using the known result
 \begin{align*}
 \frac{d^k}{dx^k}e^{ax}\sin(bx+c)=(a^2+b^2)^{k/2}e^{ax}\sin(bx+c+k\tan^{-1}(b/a)), 
 \end{align*}
we get
\begin{align*}
\frac{d^k}{dx^k}h_{\lambda}(x,t) &=\frac{e^{-\lambda t}}{\pi}\int_{0}^{\infty}\frac{e^{-ty}}{y+\lambda}\Big[e^{(\lambda^{\beta}-y^{\beta}\cos(\beta\pi))x}\left(\lambda^{2\beta}+y^{2\beta}-2\lambda^{\beta}y^{\beta}\cos(\beta\pi)\right)^{k/2}\\
 &\left[\lambda^{\beta}\sin(xy^{\beta}\sin(\beta\pi+k\alpha))+y^{\beta}\sin(xy^{\beta}\sin(\beta\pi)-\beta\pi+k\alpha)\right]\Big]dy,
\end{align*}
where
\begin{equation*}
\tan \alpha = \frac{y^{\beta}\sin(\beta\pi)}{\lambda^{\beta}-y^{\beta}\cos(\beta\pi)},~~0<\beta <1.
\end{equation*}
Using DCT, we have 
\begin{equation*}
 \lim_{x\rightarrow 0^{+}}\frac{d^k}{dx^k}h_{\lambda}(x,t) = \frac{d^k}{dx^k}h_{\lambda}(x,t)\Big|_{x=0}
\end{equation*}
which leads to the result.

\noindent (c) The result follows by induction. Using \eqref{ch3-prop2.1}, we have
 \bea
 \tilde{h}_{\lambda}(x,s) = \frac{1}{s}\Big((s+\lambda)^{\beta}-\lambda^{\beta}\Big)e^{-x((s+\lambda)^{\beta}-\lambda^{\beta})}.
 \eea
 For $m=2$,
 \beq\label{ts1}
 \frac{\partial}{\partial x}\tilde{h}_{\lambda}(x,s) = -\Big((s+\lambda)^{1/2}-\lambda^{1/2}\Big)\tilde{h}_{\lambda}(x,s)
 \eeq
 and
 \beq\label{ts2}
 \frac{\partial^2}{\partial x^2}\tilde{h}_{\lambda}(x,s) = \Big((s+\lambda)^{1/2}-\lambda^{1/2}\Big)^2\tilde{h}_{\lambda}(x,s).
 \eeq
 Using \eqref{ts1} and \eqref{ts2}, we get
 \begin{align*}
 \left( \frac{\partial^2}{\partial x^2}-2\lambda^{1/2}\frac{\partial}{\partial x}\right)\tilde{h}_{\lambda}(x,s) =\Big(s\tilde{h}_{\lambda}(x,s)-h_{\lambda}(x,0)\Big)+ h_{\lambda}(x,0).
 \end{align*}
Invert the LT to get
\bea
\left( \frac{\partial^2}{\partial x^2}-2\lambda^{1/2}\frac{\partial}{\partial x}\right)h_{\lambda}(x,t) = \frac{\partial}{\partial t}h_{\lambda}(x,t) + h_{\lambda}(x,0)\delta_0(t).
\eea
Similarly, for $m=3$
\bea
\left(\frac{\partial^3}{\partial x^3}-3\lambda^{1/3}\frac{\partial^2}{\partial x^2}+ 3\lambda^{2/3}\frac{\partial}{\partial x}\right)h_{\lambda}(x,t) = (-1)^3\left(\frac{\partial}{\partial t}h_{\lambda}(x,t) + h_{\lambda}(x,0)\delta_0(t)\right).
\eea
The result now follows in a similar manner for a general $k$.
\end{proof}

\noindent In particular for $\lambda=0$, we have the following corollary due to Hahn {\it et al.} (2011).
\begin{corollary}
(a) For $0<\beta<1$,
$$\lim_{x\rightarrow 0^{+}} h_0(x,t) = \frac{t^{-\beta}}{\Gamma(1-\beta)}.$$

\noindent(b)  For $0<\beta \leq 1/2,~~k=0,1,\cdots, \left[\frac{1}{\beta}-1\right]-1$
$$\frac{d^k}{dx^k}h_0(x,t)\Big|_{x=0} = (-1)^k\frac{t^{-(k+1)\beta}}{\Gamma(1-(k+1)\beta)}.$$

\noindent(c) For $\beta=1/m,$ $m\geq 2$,
$$(-1)^m\frac{\partial^m}{\partial x^m}h_0(x,t) = \frac{\partial}{\partial t}h_0(x,t).$$
(d) For $\beta=1/m,$ $m\geq 2$
$$\frac{\partial^{m-1}}{\partial x^{m-1}}h_0(x,t)\Big|_{x=0} = 0.$$
\end{corollary}
\begin{proof}
(a) For $\lambda=0,$ the density function $h_{\lambda}(x,t)$ reduces to the density of an inverse stable density. Put $\lambda=0$ in \eqref{eq4.5} to get 
\begin{align*}
 \lim_{x\rightarrow 0^{+}} h_0(x,t) = h_0(0,t) &= \frac{\sin(\beta\pi)}{\pi}\int_{0}^{\infty}y^{\beta-1}e^{-ty}dy\\
 & = \frac{\sin(\beta\pi)}{\pi}\frac{\Gamma{(\beta)}}{t^{\beta}}=\frac{t^{-\beta}}{\Gamma(1-\beta)},
\end{align*}
using the Euler's identity $\Gamma{(z)}\Gamma{(1-z)} = \pi/\sin{(\pi z)}.$\\

\noindent(b) Puttin $\lambda=0$ in \eqref{prop2.1b}, we get
$\tan \alpha = -\tan \beta\pi = \tan(\pi-\beta\pi)$. Taking $\alpha=\pi-\beta\pi$, we have
\begin{align*}
 \frac{d^k}{dx^k}h_0(x,t)\Big|_{x=0} &= \frac{1}{\pi}\int_{0}^{\infty}e^{-ty}y^{(k+1)\beta-1}\sin(k\alpha-\beta\pi)\\
 &=\frac{(-1)^{k+1}}{\pi}\int_{0}^{\infty}e^{-ty}y^{(k+1)\beta-1}\sin((k+1)\beta\pi)\\
 &= \frac{(-1)^{k+1}}{\pi}\frac{\Gamma((k+1)\beta)}{t^{(k+1)\beta}}\sin(k+1)\beta\pi\\
 &=(-1)^{k+1}\frac{t^{-(k+1)\beta}}{\Gamma(1-(k+1)\beta)},
\end{align*}
where $k=0,1,\cdots, \left[\frac{1}{\beta}-1\right]-1.$\\

\noindent(c) For $\lambda =0,$ \eqref{prop2.1c} reduces to
\beq\label{corol2.1c}
(-1)^m\frac{\partial^m}{\partial x^m}h_0(x,t) = \frac{\partial}{\partial t}h_0(x,t) + h_0(x,0)\delta_0(t).
\eeq
Note also that when $\lambda=0$ and $t=0$, we have
\begin{align*}
 h_0(x,0) &= \frac{1}{\pi}\int_{0}^{\infty}e^{-xy^{\beta}\cos (\beta\pi)}
 y^{\beta-1}\sin(\beta\pi- x y^{\beta}\sin(\beta\pi))dy\\
 &= \frac{1}{\beta\pi}\int_{0}^{\infty}e^{-(x\cos (\beta\pi))u}\sin(\beta\pi- x \sin(\beta\pi)u)du~~(\mbox{putting}~y^{\beta}=u)\\
 &= \frac{1}{\beta\pi}\Big[\sin(\beta\pi)\int_{0}^{\infty}e^{-(x\cos (\beta\pi))u}\cos(x \sin(\beta\pi)u)du \\
 &\hspace{1.5cm}- \cos(\beta\pi)\int_{0}^{\infty}e^{-(x\cos (\beta\pi))u}\sin(x \sin(\beta\pi)u)du\Big]\\
 &=\frac{1}{\beta\pi}\left[\sin(\beta\pi)\frac{x\cos(\beta\pi)}{x^2}-\cos(\beta\pi)\frac{x\sin(\beta\pi)}{x^2}\right]\\
 &=0.
\end{align*}
The equation before last line follows by using the fact that $\beta\leq 1/2$ for $m=2,3, \cdots,$ and the results
\bea
\int_{0}^{\infty}e^{-ax}sin(bx)dx = \frac{b}{a^2+b^2}; \ \ \int_{0}^{\infty}e^{-ax}cos(bx)dx =
\frac{a}{a^2+b^2}, \ a>0,b\in \mathbb{R}.
\eea
Thus from \eqref{corol2.1c} and using $h_0(x,0)=0$, we have the result.\\

\noindent (d) For $\beta=1/m$ and $\lambda=0$, we have
$\tan \alpha = - \tan(\pi/m)= \tan(\pi-\pi/m)$, which implies $\alpha=(m-1)/m\pi.$ Also, for $\lambda=0,$
\begin{align*}
 \frac{\partial^{m-1}}{\partial x^{m-1}}h_0(x,t)\Big|_{x=0} &= \frac{1}{\pi}\int_{0}^{\infty}\frac{e^{-ty}}{y}\left[y^{(m-1)/m}y^{1/m}\sin(-\pi/m+(m-1)\alpha)\right]dy\\
 &= \frac{1}{\pi}\int_{0}^{\infty}e^{-ty}\sin(-\pi/m+(m-1)^2/m\pi)dy\\
 &= \frac{1}{\pi}\int_{0}^{\infty}e^{-ty}\sin((m-2)\pi)dy\\
 &=0, 
\end{align*}
since $m$ is an integer $\geq 2$.

\end{proof}

\vtwo
\noindent {\bf \Large References}
\vone
\noindent
\begin{namelist}{xxx}
\item{} Abramowitz, M., Stegun, I. A. (eds), 1992. { Handbook of Mathematical Functions with Formulas, Graphs and
Mathematical Tables}. Dover, New York.

\item{}  Applebaum, D., 2009. { L\'evy Processes and Stochastic Calculus}. 2nd ed., {Cambridge University Press},
Cambridge, U.K.



\item{}  Bertoin, J., 1996. {L\'evy Processes}. {Cambridge University Press}, Cambridge.

\item{}  Cont, R. and Tankov, P., 2004. {Financial Modeling with Jump Processes}. { Chapman \& Hall CRC Press}, Boca Raton.

\item{} Decreusefond, L. and Nualart, D., 2008. Hitting times for Gaussian processes.
{Ann. Probab.} {36}, 319-330.


\item{} Hahn, M. G., Kobayashi, K.,  Umarov, S., 2011. Fokker-plank-Kolmogorov equations associated
with time-changed fractional Brownian motion.
 { Proc. Amer. Math. Soc.} {139}, 691-705.

\item{} Keyantuo, V., Lizama, C., 2012. On a connection between powers of operators and Fractional Cauchy problems. {J. Evol. Equ.} {12}, 245--265.

\item{} Lee, M-L.T. and Whitmore, G. A., 2006. Threshold regression for survival analysis: modeling event times by a stochastic process reaching a boundary. {Statistical Science}, {21}, 501-513.

\item{} Meerschaert, M. M., Scheffler, H., 2008a. Triangular array limits for continuous time random walks.
{Stochastic Process. Appl.} {118}, 1606--1633.

\item {} Meerschaert, M. M., Zhang, Y. and Baeumer, B., 2008b. Tempered anomalous diffusion in heterogeneous systems, {Geophys. Res. Lett.} {35}, p. L17403.

\item{}  Meerschaert, M. M., Nane, E., Vellaisamy, P., 2011. The fractional Poisson process and the inverse stable subordinator. {Electron. J. Probab.} {16}, 1600--1620.

\item{}  Meerschaert, M. M., Nane, E., Vellaisamy, P., 2013. Transient anomalous subdiffusions
on bounded domains. {Proc. Amer. Math. Soc.} 141, 699--710.



\item {} Rosi\'nski, J., 2007.  {Tempering stable processes}. {Stochastic Process  Appl.} { 117},  677--707.


\item{} Sato, K.-I., 1999. {L\'evy Processes and Infinitely Divisible Distributions}. Cambridge University Press.

\item{} Schiff, J. L., 1999. {The Laplace Transform: Theory and Applications}. Springer-Verlag, New York.

\item{} Stanislavsky, A.,  Weron, K. and  Weron, A., 2008. Diffusion and relaxation controlled by tempered $\alpha$-stable processes. {Phys. Rev. E.} 78, No. 5, 051106.

\item{} Uchaikin, V. V. and Zolotarev, V. M., 1999. {Chance and Stability: Stable Distributions and Their Applications}.
VSP. Utrecht.

\item{} Veillette, M. and Taqqu, M. S., 2010. Numerical computation of first-passage times
of increasing L\'evy Processes. {Methodol. Comput. Appl. Probab.}, {12}, 695--729.

\item{} Vellaisamy, P. and Kumar, A., 2013. Hitting times of an inverse Gaussian process. Submitted.
\end{namelist}

\end{document}